\documentclass[11pt, oneside]{amsart}
\usepackage[foot]{amsaddr}
\usepackage{amsfonts}
\usepackage[T1]{fontenc} 
\usepackage[utf8]{inputenc}
\usepackage[english]{babel}
\usepackage[a4paper, total={6in, 9in}]{geometry}
\usepackage{hyperref}
\usepackage{amsmath}
\usepackage{amsthm}
\usepackage{amssymb}
\usepackage{mathtools}
\usepackage{indentfirst}
\usepackage{array}
\usepackage{dsfont}

\theoremstyle{plain}
\newtheorem{theorem}{Theorem}[section]

\newtheorem{prop}[theorem]{Proposition}
\newtheorem{lemma}[theorem]{Lemma}

\theoremstyle{remark}

\newcommand{\N}{\mathbb{N}}
\newcommand{\R}{\mathbb{R}}
\newcommand{\C}{\mathbb{C}}

\renewcommand{\C}{\mathbb{C}}

\newcommand{\abs}[1]{\left|#1\right|}
\newcommand{\im}{\mathrm{Im}\,}
\newcommand{\ree}{\mathrm{Re}\,}
\newcommand{\norm}[1]{\left\lVert #1\right\rVert}

\DeclareMathOperator{\diag}{diag}

\begin{document}

\title{Bounds on Whittaker functions for $\mathrm{GL}(n)$}

\author{Gilles Felber}
\address[G. Felber and D. T\'oth]{HUN-REN Alfr\'ed R\'enyi Institute of Mathematics,
Budapest, Hungary}
\author{D\'avid
T\'oth}
\address[D. T\'oth]{
Department of Algebra and Geometry,  Institute of Mathematics,
Budapest University of Technology and Economics,
M\H{u}egyetem rkp.\ 3, H-1111 Budapest, Hungary}
\email{felber@renyi.hu, toth.david.akos@ttk.bme.hu}

\begin{abstract}
We prove a new estimate on Jacquet-Whittaker function for $\mathrm{GL}_n(\mathbb R)$, assuming that the Langlands parameters are purely imaginary and well-spaced. This gives an upper bound for the global sup-norm of the Whittaker function that matches the lower bound of Brumley-Templier in the exponent up to a linear error in $n$.
\end{abstract}

\maketitle

\section{Introduction}\label{sec:introduction}

In this paper, we study
the (standard archimedean spherical)
Jacquet–Whittaker function $\mathcal W_\mu$ on $\mathrm{GL}_n(\R)$ associated with a 
cusp form $\phi$
and indexed by the Langlands parameters
\begin{equation}\label{eq:Langlands_params}
\mu=(\mu_1,\dots,\mu_n)\in \C^n\quad\textrm{ with }\quad\mu_1+\dots+\mu_n=0,    
\end{equation}
 so that the automorphic representation induced by $\phi$ is tempered,
 yielding that
 \begin{equation}\label{eq:Langlands_params_tempered}
     \mu=(\mu_1,\dots,\mu_n)\in (i\R)^n
 \end{equation}
 holds as well.
We will also assume that the parameters are well-spaced, i.e.\ 
\begin{equation}\label{eq:Langlands_params_well_spaced}
|\mu_i-\mu_j|\geq cT_\mu\qquad (1\leq i<j\leq n)    
\end{equation}
holds for some constant $c>0$, where 
\begin{equation}\label{eq:T_mu_def}
T_\mu:=\max(2,|\mu_1|,\dots,|\mu_n|).
\end{equation}
This covers $99\%$ of all Maass forms (see the introduction of \cite{BB20} for more details).
The constant $c$ will be implicit in the rest of the article and we will use the Vinogradov notation $\gg$ to indicate 
\eqref{eq:Langlands_params_well_spaced}
from now on.  Recall that $T_\mu\asymp_n\lambda_\phi^{1/2}$ is related to the Laplace eigenvalue of the cusp form $\phi$ \cite[Equation (7)]{BHM20}.

The function $\mathcal{W}_\mu$ is defined as follows. By the Iwasawa decomposition,
every matrix $g\in\mathrm{GL}_n(\R)$ can be uniquely written as $g=utk$, where
$u\in \mathrm{U}_n(\R)$ is unipotent upper-triangular,
$t=\mathrm{diag}(t_1,\dots,t_n)$ is diagonal with positive diagonal entries,
and $k\in\mathrm{O}_n(\R)$ is orthogonal.
The height function $H_\mu$ is given for every $\mu\in\C^n$ by
\[
H_\mu(g)=\prod_{j= 1}^nt_j^{(n+1)/2-j+ \mu_j}\quad (g\in\mathrm{GL}_n(\R)).
\]
It is right-invariant 
under $\mathrm O_n(\R)$ by the Iwasawa decomposition, and if in
addition \eqref{eq:Langlands_params} 
is satisfied, 
then it is also invariant under $\mathrm Z_n(\R)$ (the center of $\mathrm{GL}_n(\R)$), so
it can be regarded as a function on 
the generalized upper half-plane
$\mathfrak h^n = \mathrm{GL}_n(\R)/(\mathrm O_n(\R)\cdot \mathrm Z_n(\R))$.
We define the Jacquet–Whittaker function by
\begin{equation}\label{eq:Whittaker_integral}
\mathcal W_\mu(g)=\int_{\mathrm U_n(\R)}
H_\mu(wug)\overline{\psi(u)}\,du
\end{equation}
where  $w$ is the long Weyl element and
\[
\psi(u)=\exp(2\pi i (u_{1,2}+\dots+u_{n-1,n})), \quad u = (u_{ij})\in \mathrm  U_n(\R).
\]
The integral in \eqref{eq:Whittaker_integral}
converges locally uniformly for any $g\in \mathrm{GL}_n(\R)$ and for any $\mu\in\C^n$ such
that $\ree\mu_1>\dots>\ree \mu_n$, and hence defines a holomorphic function $\mu\mapsto \mathcal W_\mu(g)$.
 Jacquet~\cite{Ja} proved the holomorphic continuation of this function
 for any $\mu\in\C^n$. 
 
 All this implies that  $\mathcal W_\mu(g)$  is invariant under $\mathrm Z_n(\R)$
 once \eqref{eq:Langlands_params}
 is satisfied,
right-invariant under $\mathrm O_n(\R)$
and  transforms by
a character under the left action of $\mathrm U_n(\R)$,
hence it is enough to estimate it 
 on the positive diagonal torus.
 Note that Jacquet \cite{Ja} also proved that
 the completed  Jacquet–Whittaker function
 \[
W_\mu(g):=\left(\prod_{1\leq j<k\leq n}\Gamma_\R(1+\mu_j-\mu_k)\right)\mathcal{W}_\mu(g)\,\quad
 \]
 is invariant under any permutation of the $\mu_j$'s
 (here $\Gamma_\R(s)=\pi^{-s/2}\Gamma(s/2)$), hence
 we may (and will) assume
 \[
\im \mu_1\geq \dots \geq \im \mu_n
\]
as well.
 We are going to show the following:

\begin{theorem}\label{thm:main_thm}
Let $n\geq3$, $t=\mathrm{diag}(t_1,\dots,t_n)\in \mathrm{GL}_n(\R)$ with
$t_1,\dots,t_n>0$, and assume that the Langlands parameters 
satisfy \eqref{eq:Langlands_params},
\eqref{eq:Langlands_params_tempered} and \eqref{eq:Langlands_params_well_spaced}. 
Then for any $\epsilon>0$ we have
\begin{equation}\label{eq:main_thm_eq}
\mathcal W_\mu(\diag(t_1,\dots,t_{n}))
\ll_{n,\epsilon} 
T_\mu^{-\frac{3n^2-11n+14}{12}+\epsilon}
\left(\prod_{j=1}^nt_j^{n+1-2j}\right)^{1/2-\epsilon}
\exp \left(-\frac{1}{T_\mu}\sum_{j=1}^{n-1} t_j/t_{j+1}\right).
\end{equation}
\end{theorem}

Note that as $\mathcal W_\mu$ is invariant under $\mathrm Z_n(\R)$, it
is clearly enough to show \eqref{eq:main_thm_eq} for $t=\mathrm{diag}(t_1,\dots,t_{n-1},1)$,
or equivalently, using the parametrization of \cite{BrTe, Sta90, Sta01, Sta02},
it suffices to prove
\[
\mathcal W_\mu(\diag(y_1\dots y_{n-1},\dots,y_{n-1},1))
\ll_{n,\epsilon}
T_\mu^{-\frac{3n^2-11n+14}{12}+\epsilon}
\prod_{j=1}^{n-1}y_j^{\frac{j(n-j)}2-\epsilon}
\exp \left(-\frac{1}{T_\mu}\sum_{j=1}^{n-1} y_j\right).
\]
This is the estimate that we are going to prove in 
Section~\ref{sec:proof}.  
Assumption \eqref{eq:Langlands_params_well_spaced} is only needed in Section \ref{sec:conclusion_of_the_proof}, when we estimate the Gamma factors. A bound involving quotients of Gamma functions can 
be deduced assuming only that $|\mu_1-\mu_n|\gg1$.

As in  \cite[Theorem 1]{BHM20}, \eqref{eq:main_thm_eq} implies
\begin{equation}\label{eq:infty_norm_upper_bound}
    \|\mathcal W_\mu\|_\infty\ll_{n,\epsilon}
T_\mu^{\frac{(n^3-n)-(3n^2-11n+14)}{12}+\epsilon}=T_\mu^{\frac{n^3-3n^2+10n-14}{12}+\epsilon}.
\end{equation}
This complements \cite[Theorem 1.14]{BrTe} that states 
(under the same conditions)
\begin{equation}\label{eq:BT_lower_bound}
\|\mathcal W_\mu\|_\infty\gg_{n} T_\mu^{n(n-1)(n-2)/12}.    
\end{equation}
Note that the cubic and the quadratic terms of the exponent 
agree in the upper and lower bounds above. 
More precisely, the difference in the exponent between our result and \cite{BrTe} is 
$\frac23(n-\frac74)$.
Also, the upper bound 
 \eqref{eq:infty_norm_upper_bound} improves on the 
 one given in \cite[Remark 1]{BHM20} if $n\geq4$,
 and  the main 
 ingredient of
 our improvement is the use of a Mellin transform computed by Stade (see Equation \eqref{eq:Stade}).  
 For $n=3$ the two bounds are the same,
 and in this case the improved bound 
 $\norm{\mathcal W_\mu}_\infty\gg T_\mu^{3/4}$ is stated in \cite[Theorem 1.9]{BrTe}   
while  the upper estimate $\norm{\mathcal W_\mu}_\infty\ll T_\mu$ is given in
\cite[Lemma 2]{BHM19}
(while we get $\norm{\mathcal W_\mu}_\infty\ll T_\mu^{4/3}$). We expect the optimal exponent for $n\geq4$ to sit strictly between the two bounds \eqref{eq:infty_norm_upper_bound} and \eqref{eq:BT_lower_bound}.

One of the main interest of global bounds on Whittaker functions is the global sup-norm problem for cusp forms on $\mathrm{GL}_n(\R)$. The best known lower bounds for this problem are given by \cite{BrTe} for $n\geq3$ and the best known upper bounds can be found in \cite{BHM19} for $n=3$ and \cite{BHM20} for $n\geq4$. See also the citations therein for a more detailed account of the literature. Following the proof of \cite{BHM20}, we can give a modest improvement on the bound of \cite[Theorem 4]{BHM20} for the global sup-norm problem for cusp forms on $\mathrm{GL}_n(\R)$. The improvement is only in the quadratic term of the exponent. It would be of great interest to lower their bound in the cubic term and match the corresponding term of the lower bound. One approach would be to refine the counting argument in Section 3.2 of \cite{BHM20}. Another interesting question is the case of non-tempered automorphic representations. In that case, we know that the real parts of the components of $\mu$ are strictly between $-\frac12$ and $\frac12$ (see \cite[Equation (5)]{BHM20}). It would be of interest to improve the bound of \cite[Theorem 2]{BHM20} in that generality. We hope to return to these problems soon.

\subsection{Acknowledgment}
The authors are profoundly thankful to Péter Maga, who introduced them to this topic and guided them in this project. The research towards this paper was supported by the MTA–RI Lendület “Momentum” Analytic Number Theory and Representation Theory Research Group for both authors.

\section{Bounds on integrals of \texorpdfstring{$K$}{K}-Bessel functions}

In this section we prove an estimate for the $K$-Bessel functions that will be used frequently 
in the proof of Theorem~\ref{thm:main_thm}.  In the course of the proof we are going to the need the following:
\begin{prop}
Let $a>0$, $b\in\R$ and $x>0$. Assume that $b\leq1$ or $ax\gg1$, then we have
\begin{align}\label{eq bound partial Gamma}
\int_x^\infty e^{-at}t^b\frac{dt}t\ll_be^{-ax}a^{-1}x^{b-1}.
\end{align}    
\end{prop}
\begin{proof}
If $b\leq1$, then
\[
\int_x^\infty e^{-at}t^b\frac{dt}t\leq x^{b-1}\int_x^\infty e^{-at}dt\leq x^{b-1}a^{-1}e^{-  
ax}.
\]
If $b>1$ and $ax\gg 1$, then the statement can be proved by induction using 
integration by parts.
\end{proof}

\begin{lemma}\label{lem bound integral of Bessel function}
Let $r\gg1$, $y>0$ and $\ell,k>0$ fixed.
Then we have
$$\int_0^\infty\abs{K_{ir}(2\pi y\sqrt{1+u^{\pm 2}})}^\ell u^{\pm k}\frac{du}u\ll_{\ell,k}e^{-\pi\ell r/2}r^{k-\ell/3}y^{-k}.$$
Moreover, if $2\pi y>1+\pi r/2$, then
$$\int_0^\infty\abs{K_{ir}(2\pi y\sqrt{1+u^{\pm 2}})}^\ell u^{\pm k}\frac{du}u\ll_{\ell,k}e^{-\pi\ell y}y^{-k-\ell/2}.$$
If in addition $r\gg T_\mu$ holds, where $T_\mu$ is defined in \eqref{eq:T_mu_def},
then we also have the uniform estimate
$$\int_0^\infty\abs{K_{ir}(2\pi y\sqrt{1+u^{\pm 2}})}^\ell u^{\pm k}\frac{du}u\ll_{\ell,k}e^{-\ell y/T_\mu}e^{-\pi\ell r/2}r^{k-\ell/3}y^{-k}.$$
\end{lemma}

\begin{proof}
We use the following bounds for the $K$-Bessel function 
of imaginary argument:
if $x>0$, $r\geq 1$, then we have (see \cite[p.\ 679]{BH09})
\begin{align*}
e^{\pi r/2}\abs{K_{ir}(x)}\ll 
\min(r^{-1/3},
    \abs{r^2-x^2}^{-1/4}),
\end{align*}
and so
\begin{equation}\label{eq:implied bound}
K_{ir}(x)\ll e^{-\pi r/2}r^{-1/3}
\end{equation}
holds for every $x>0$ and $r\geq1$. On the other hand, if
$x>1+\pi r/2$, then the estimate
\begin{equation}\label{eq bounds K Bessel 2}
\abs{K_{ir}(x)}\ll e^{-x}x^{-1/2}
\end{equation}
holds  by \cite[Proposition 9]{HM06} for all $r\in\R$.

If $x>0$ but $0<\delta<r<1$, then the function $K_{ir}(x)$ is uniformly bounded by a constant only depending on 
$\delta$.
We show this by splitting the usual integral representation for $K_{ir}(x)$:
\begin{align*}
K_{ir}(x)&=\int_0^\infty e^{-x\cosh(t)}\cos(rt)dt\\
    &=\int_0^1e^{-x\cosh(t)}\cos(rt)dt+\int_1^\infty e^{-x\cosh(t)}\cos(rt)dt\\
    &\ll1+\left.e^{-x\cosh(t)}\frac{\sin(rt)}r\right|_1^\infty+\frac xr\int_1^\infty e^{-x\cosh(t)}\sinh(t)\sin(rt)dt\\
    &\ll1+\frac1r+\frac xr\int_1^\infty e^{-x\cosh(t)}\sinh(t)dt\\
    &\ll1+\frac1r-\left.\frac xr\frac{e^{-x\cosh(t)}}x\right|_1^\infty\\
    &\ll1+\frac1r+\frac1r.
\end{align*}
This shows that Equation \eqref{eq:implied bound} is valid for all $x>0$ and $r\gg1$
(where the implied constant depends on $\delta$).

Turning to the proof of the statements, first note that it is enough to prove them with positive signs in the exponents and then 
the change of variables $u\mapsto u^{-1}$ gives the result with negative signs. 

 For the first bound we split the integral at the point $u=\frac{\pi r+2}{4\pi y}$:
\begin{align*}
\int_0^\infty&\abs{K_{ir}(2\pi y\sqrt{1+u^2})}^\ell u^k\frac{du}u=\\
&=\int_0^{\frac{\pi r+2}{4\pi y}}\abs{K_{ir}(2\pi y\sqrt{1+u^2})}^\ell u^k\frac{du}u+\int_{\frac{\pi r+2}{4\pi y}}^\infty\abs{K_{ir}(2\pi y\sqrt{1+u^2})}^\ell u^k\frac{du}u=:I_1+I_2.
\end{align*}
In the case of $I_1$,
we apply the bound \eqref{eq:implied bound} yielding
$$\int_0^{\frac{\pi r+2}{4\pi y}}\abs{K_{ir}(2\pi y\sqrt{1+u^2})}^\ell u^k\frac{du}u\ll_\ell
e^{-\pi\ell r/2}r^{-\ell/3}\int_0^{\frac{\pi r+2}{4\pi y}}u^k\frac{du}u
\ll_{\ell,k} e^{-\pi\ell r/2}r^{k-\ell/3}y^{-k}.$$

Now consider $I_2$. In that case, we have $2\pi y\sqrt{1+u^2}\geq\frac{\pi r+2}2$. Using first  \eqref{eq bounds K Bessel 2} and then \eqref{eq bound partial Gamma}, we get:
\begin{align*}
I_2&\ll_\ell\int_{\frac{\pi r+2}{4\pi y}}^\infty e^{-2\pi\ell y\sqrt{1+u^2}}\left(y\sqrt{1+u^2}\right)^{-\ell/2}u^k\frac{du}u\\
&\ll_\ell y^{-\ell/2}\int_{\frac{\pi r+2}{4\pi y}}^\infty e^{-2\pi\ell yu}u^{k-\ell/2}\frac{du}u\\
&\ll_{\ell,k}e^{-\pi\ell r/2}y^{-\ell/2}y^{-1}(r/y)^{k-\ell/2-1} =
e^{-\pi\ell r/2}r^{k-\ell/2-1}y^{-k} .
\end{align*}

If $y\geq\frac{\pi r+2}{4\pi}$, then using $\sqrt{1+u^2}\geq\frac{1+u}2$
the corresponding bound is an easy consequence of \eqref{eq bounds K Bessel 2}:
\begin{align*}
\int_0^\infty\abs{K_{ir}(2\pi y\sqrt{1+u^2})}^\ell u^k\frac{du}u&\ll_\ell\int_0^\infty e^{-\pi\ell y(1+u)}y^{-\ell/2}u^k\frac{du}u\\
    &\ll_{\ell,k}e^{-\pi\ell y}y^{-k-\ell/2}\int_0^\infty e^{-t}t^{k}\frac{dt}t\\
    &\ll_{\ell,k}e^{-\pi\ell y}y^{-k-\ell/2}.
\end{align*}
We conclude with the proof of the last bound, so assume that $r\gg T_\mu$ holds. 
Since $e^{-\ell y/r}$ is bounded from below for 
$ y\leq 1+ r$, so is $e^{-\ell y/ T_\mu}$. Hence the first bound implies the third one. 
If  $y>1+r$, then $2\pi y>1+\pi r/2$ and $y^{-\ell/2}\ll r^{-\ell/3}$. Now the third bound follows from
the second one:
\begin{align*}
\int_0^\infty\abs{K_{ir}(2\pi y\sqrt{1+u^{\pm 2}})}^\ell u^{ k}\frac{du}u&\ll_{\ell,k}e^{-\pi\ell y}y^{-k-\ell/2}\\
    &\ll_{\ell,k}e^{-\pi\ell y/2}e^{-\pi\ell r/2}r^ky^{-k}r^{-\ell/3}\\
     &\ll_{\ell,k}e^{-\ell y/T_\mu}e^{-\pi\ell r/2}r^{k-\ell/3}y^{-k},
\end{align*}
since $T_\mu\geq2$ by definition and $r\gg1$.
\end{proof}

\section{Proof of the main theorem}\label{sec:proof}

Before turning to the argument we state the following formula of Stade that will serve as one of our tools (see \cite[Equation (33)]{BHM20}, \cite[Theorem 1.1]{Sta02}). Using
the notations of Section~\ref{sec:introduction}, for an appropriate $s\in\C$ we have
\begin{align}\label{eq:Stade}
\int_{\R_{>0}^{n-1}}\abs{W_\mu(\diag(t_1,\dots,t_{n-1},1))}^2\prod_{j=1}^{n-1}\frac{t_j^{s-1}\,dt_j}{t_j^{n+1-2j}}=\frac{2^{1-n}}{\Gamma_\R(ns)}\prod_{k,l=1}^n\Gamma_\R(s+\mu_k-\mu_l).
\end{align}
To handle the convergence of the integral on the left-hand side we change to $y$-notation, i.e., 
make the change of variables  $y_j\cdots y_{n-1}=t_j$ ($1\leq j\leq n-1$). The determinant of the 
(upper-triangular) Jacobian is $\prod_{j=1}^{n-1}y_j^{j-1}$, and  left-hand side of \eqref{eq:Stade}
is
\begin{align*}
\int_{\R_{>0}^{n-1}}\abs{W_\mu(\diag(y_1\cdots y_{n-1},\dots,y_{n-1},1))}^2\prod_{j=1}^{n-1}y_j^{j(s+j-n)}\frac{dy_j}{y_j}.
\end{align*}
By \cite[Equation (14)]{BHM20}, we have 
\[
\abs{W_\mu(\diag(y_1\dots y_{n-1},\dots,y_{n-1},1))}\ll_{n,\mu,\epsilon}
\left(\prod_{j=1}^{n-1}y_j^{j(n-j)}\right)^{1/2-\epsilon}
    \exp\left(-\frac1 {T_\mu}\sum_{j=1}^{n-1}y_j\right)
\]
for any $\epsilon>0$, where $T_\mu=\max(2,|\mu_1|,\dots,|\mu_n|)$. Hence the formula holds if the integral of
\[
\prod_{j=1}^{n-1} y_j^{j(n-j)(1-\epsilon)+j(s+j-n)-1}=
\prod_{j=1}^{n-1}y_j^{js-1-\epsilon j(n-j)}
\]
converges in a neighborhood of 0. 
Choosing an appropriate $\epsilon>0$ we get that $\eqref{eq:Stade}$ holds if $j\ree s-1>-1$, i.e.,\ when $\ree s>0$.

The starting point of our proof is a  recursion formula of  Stade.
Before stating it we introduce the notation
\begin{align*}
&W_\mu^*(y_1,\dots,y_{n-1})=\\
&\qquad W_\mu(\mathrm{diag}(y_1\dots y_{n-1},y_{n-2}\dots y_{n-1},\dots,y_{n-1},1))
\prod_{i=1}^{n-1}y_i^{-\frac{i(n-i)}{2}}
\prod_{j=1}^iy_i^{-\frac{\mu_j+\mu_{n+1-j}}{2}},
\end{align*}
then we have the following (see \cite[Equation (32)]{BHM20},  \cite[Equation (4.3)]{Sta01}):
\begin{theorem}
If $n\geq 2$, then
\begin{align*}
W_\mu^\ast(y_1,\dots,y_{n-1})&=2^{2n-3}\int_{\R_{>0}^{n-2}}\prod_{j=1}^{n-1}u_j^{\frac{\mu_{j+1}+\mu_{n-j}-\mu_1-\mu_n}2}K_{\frac{\mu_1-\mu_n}2}\left(2\pi y_j\sqrt{(1+u_{j-1}^2)(1+u_j^{-2})}\right)\\
	&\qquad\qquad\qquad\qquad\qquad\cdot W_{\mu'}^\ast\left(y_2\frac{u_1}{u_2},\dots,y_{n-2}\frac{u_{n-3}}{u_{n-2}}\right)\frac{du_1\cdots du_{n-2}}{u_1\cdots u_{n-2}}
\end{align*}
where
$$\mu'=\left(\mu_2+\frac{\mu_1+\mu_n}{n-2},\dots,\mu_{n-1}+\frac{\mu_1+\mu_n}{n-2}\right)\in\C^{n-2}$$
$u_0=u_{n-1}^{-1}=0$, $u_{n-1}^0=1$, and for $n=2,3$ the inner Whittaker function
$W^\ast_{\mu'}$ is understood to be equal $1$.
\end{theorem}
 The outline of the proof is the following:
 we apply the Cauchy-Schwarz inequality to separate the inner Whittaker function 
 on the right-hand side of the above formula. 
 To ensure convergence we also include  one Bessel function
 and a product  $P(u_1,\dots,u_{n-2})$ of powers of $u_j$ next to it so
 that $2^{-4n+6}|W_\mu^\ast(y_1,\dots,y_{n-1})|^2$ is bounded by 
\begin{align*}
	&\left(\int_{\R_{>0}^{n-2}}\prod_{j=1}^{n-2}\abs{K_j}^2\abs{K_{n-1}}P(u_1,\dots,u_{n-2})^{-1}\right)
    \left(\int_{\R_{>0}^{n-2}}\abs{W_{\mu'}^{\ast(n-2)}}^2\abs{K_{n-1}}P(u_1,\dots,u_{n-2})\right)\nonumber
\end{align*}
(written briefly)
where we also use that $\ree \mu_j=0$ for $1\leq j\leq n$.
Here and later we indicate the rank
in the notation of the Whittaker functions.
Note that the factor
$$K_{n-1}=K_{\frac{\mu_1-\mu_n}2 }\left(2\pi y_{n-1}\sqrt{1+u_{n-2}^2}\right)$$
only depends on
 $u_{n-2}$ and decays exponentially as $u_{n-2}\to\infty$.
For the right factor of the previous product we will apply \eqref{eq:Stade}, while
in the case of the left factor we will use the Cauchy-Schwarz inequality and induction on $n$. 
We will use the same power for each $u_j$, that is,
we set $P(u_1,\dots,u_{n-2})=\prod_{j=1}^{n-2}u_j^s$. 

From now on, we write $r=\im(\frac{\mu_1-\mu_n}2)$. We have $r\gg1$ where we emphasize that the implied constant 
depends directly on the condition of Equation \eqref{eq:Langlands_params_well_spaced}.

\subsection{Estimation of the right factor}
Let us set
\begin{align*}
I_R:&=
\int_{\R_{>0}^{n-2}}
\abs{W_{\mu'}^{\ast(n-2)}\left(y_2\frac{u_1}{u_2},\dots,y_{n-2}\frac{u_{n-3}}{u_{n-2}}\right)}^2
\abs{K_{ir}\left(2\pi y_{n-1}\sqrt{1+u_{n-2}^2}\right)}
    \prod_{j=1}^{n-2}u_j^{s}\frac{du_j}{u_j}
\end{align*}
We are going to show the following:
\begin{prop}\label{prop:right_factor}
    If $s>0$, then
\begin{align*}
I_R&\ll_{n,s}
\prod_{j=1}^{n-1}y_j^{-(j-1)s}\prod_{k,l=2}^{n-1}\Gamma_\R(s+\mu_k-\mu_l)e^{-\pi r/2-y_{n-1}/T_\mu}r^{(n-2)s-1/3}.
\end{align*}
\end{prop}

For the proof we first change to the function
$W_{\mu'}^{(n-2)}$ from $W^{*(n-2)}_{\mu'}$ and also make a change of variables, i.e., we set 
\[
t_j=\prod_{k=j}^{n-3}y_{k+1}\frac{u_k}{u_{k+1}}=y_{j+1}\dots y_{n-2}\frac{u_j}{u_{n-2}}
\]
for $j=1,\dots,n-2$ with the convention $t_{n-2}=1$ (meaning that the value of the
empty product is $1$).
We also have
\[
u_j=\frac{t_ju_{n-2}}{y_{j+1}\dots y_{n-2}}\quad \quad \quad (1\leq j\leq n-2).
\]
Note that we do not apply a change of variable to $u_{n-2}$. Then $I_R$ can be written as
\begin{align*}
&\prod_{j=1}^{n-3}(y_{j+1}\cdots y_{n-2})^{-s}\int_{\R_{>0}^{n-2}}\abs{W_{\mu'}^{(n-2)}\left(\diag(t_1,\dots,t_{n-3},1)\right)}^2\abs{K_{ir}\left(2\pi y_{n-1}\sqrt{1+u_{n-2}^2}\right)}\\
	&\qquad\qquad\qquad\qquad\quad\quad\qquad\qquad\times\prod_{j=1}^{n-3}
        \left[
            \left(\frac{t_j}{t_{j+1}}\right)^{-j(n-2-j)}
            (t_ju_{n-2})^{s}\frac{dt_j}{t_j}
        \right]
    u_{n-2}^{s}\frac{du_{n-2}}{u_{n-2}}\\
	&=\prod_{j=2}^{n-2}y_j^{-(j-1)s}\int_{\R_{>0}^{n-3}}
    \abs{W_{\mu'}^{(n-2)}\left(\diag(t_1,\dots,t_{n-3},1)\right)}^2
    \prod_{j=1}^{n-3}t_j^{s+2j-n+1}\frac{dt_j}{t_j}\\
	&\qquad\qquad\qquad\qquad\quad\qquad\qquad\quad\times\int_{\R_{>0}}\abs{K_{ir}\left(2\pi y_{n-1}\sqrt{1+u_{n-2}^2}\right)}u_{n-2}^{(n-2)s}\frac{du_{n-2}}{u_{n-2}}.
\end{align*}
Setting $Y:=\prod_{j=2}^{n-2}y_j^{-(j-1)s}$ and applying  \eqref{eq:Stade} 
for the first integral (with $W_{\mu'}^{(n-2)}$ 
in place of $W_\mu^{(n)}$)
we obtain
\begin{align*}
I_R&=\frac{2^{3-n}\cdot Y}{\Gamma_\R((n-2)s)}\prod_{k,l=2}^{n-1}\Gamma_\R(s+\mu_k-\mu_l)\int_{0}^\infty\abs{K_{ir}\left(2\pi y_{n-1}\sqrt{1+u_{n-2}^2}\right)}u_{n-2}^{(n-2)s}\frac{du_{n-2}}{u_{n-2}}.
\end{align*}
Proposition~\ref{prop:right_factor} follows now from the third estimate
of Lemma~\ref{lem bound integral of Bessel function}.

\subsection{Estimation of the left factor}

Let us set 
\begin{align*}
I_L&=\int_{\R_{>0}^{n-2}}\prod_{j=1}^{n-2}
K_{ir}\left(2\pi y_j\sqrt{(1+u_{j-1}^2)(1+u_j^{-2})}\right)^2
\abs{K_{ir}\left(2\pi y_{n-1}\sqrt{1+u_{n-2}^2}\right)}
\prod_{j=1}^{n-2}u_j^{-s}\frac{du_j}{u_j}.
\end{align*}
First we isolate the integral over $u_{n-2}$ and apply the Cauchy-Schwarz inequality:
\begin{align*}
\int_0^\infty&K_{ir}\left(2\pi y_{n-2}\sqrt{(1+u_{n-3}^2)(1+u_{n-2}^{-2})}\right)^2\abs{K_{ir}\left(2\pi y_{n-1}\sqrt{1+u_{n-2}^2}\right)}u_{n-2}^{-s}\frac{du_{n-2}}{u_{n-2}}\\
    &\qquad\ll\left(\int_0^\infty 
        K_{ir}\left(2\pi y_{n-2}\sqrt{(1+u_{n-3}^2)(1+u_{n-2}^{-2})}\right)^4
        u_{n-2}^{-2s-\gamma}\frac{du_{n-2}}{u_{n-2}}\right)^{1/2}\\
    &\qquad\qquad\qquad\qquad\qquad\quad \times
    \left(\int_0^\infty K_{ir}\left(2\pi y_{n-1}\sqrt{1+u_{n-2}^2}\right)^2u_{n-2}^\gamma\frac{du_{n-2}}{u_{n-2}}\right)^{1/2}.
\end{align*}
We assume  that $\gamma>0$  and keep our previous assumption $s>0$ so that $2s+\gamma>0$ holds,
and apply
 Lemma~\ref{lem bound integral of Bessel function} on the right-hand side:
 in the case of the first factor we have the parameters $\ell =4$, $k=2s+\gamma$,
 $y=y_{n-2}\sqrt{1+u_{n-3}^2}$, while
 in the case of the second factor we have
 $\ell =2$, $k=\gamma$,
 $y=y_{n-1}$.
 Then the integral over $u_{n-2}$ is bounded by
\begin{align*}
&e^{-\pi r -\frac{2y_{n-2}\sqrt{1+u_{n-3}^2}}{T_\mu}}
    r^{s+\gamma/2-2/3}\left(y_{n-2}\sqrt{1+u_{n-3}^2}\right)^{-s-\gamma/2}
\cdot e^{-\frac{\pi r}{2} -\frac{y_{n-1}}{T_\mu}}
    r^{\gamma/2-1/3}y_{n-1}^{-\gamma/2}=\\[3mm]
    &\qquad 
    = \exp\left(-\frac{3\pi r}{2}-\frac{2y_{n-2}\sqrt{1+u_{n-3}^2}}{T_\mu}-\frac{y_{n-1}}{T_\mu}\right)
    r^{s+\gamma-1}y_{n-2}^{-s-\gamma/2}y_{n-1}^{-\gamma/2}
    (1+u_{n-3}^2)^{-\frac{2s+\gamma}{4}}.
\end{align*}
We estimate this further by estimating $1+u_{n-3}^2$ from below. 
We use  $1+u_{n-3}^2\geq 1$ in the exponential while
in the case $n>3$  estimate the last factor  by $1+u_{n-3}^2\geq u_{n-3}^2$
to obtain the bound $B_{n-2}u_{n-3}^{-s-\gamma/2}$, where
\begin{equation}\label{eq:B_n-2}
B_{n-2}:= \exp\left(-\frac{3\pi r}{2}-\frac{2y_{n-2}+y_{n-1}}{T_\mu}\right)
    r^{s+\gamma-1}y_{n-2}^{-s-\gamma/2}y_{n-1}^{-\gamma/2}.
    \end{equation}
 Let us summarize this
in the following estimates: if $n>3$, then
\begin{align}\label{eq:I_L_bound2}
I_L&\ll_{s,\gamma}   B_{n-2}\int_{\R_{>0}^{n-3}}\prod_{j=1}^{n-3}K_{ir}\left(2\pi y_j\sqrt{(1+u_{j-1}^2)(1+u_j^{-2})}\right)^2\left(\prod_{j=1}^{n-4}u_j^{-s}\frac{du_i}{u_i}\right)u_{n-3}^{-2s-\frac \gamma2}\frac{du_{n-3}}{u_{n-3}}.
\end{align}
Note that in the case $n=3$ our bound for $I_L$ is simply $B_{n-2}$.
We now give an upper bound for the latter integral  by induction.
\begin{lemma}
Let $n\in\N^+$, $r\geq1$ and $y_1,\dots,y_n>0$.
If
$\sum_{m=j}^nk_m$ is positive, then
\begin{align*}
\int_{\R_{>0}^n}K_{ir}&\left(2\pi y_1\sqrt{1+u_1^{-2}}\right)^2
\prod_{j=2}^nK_{ir}\left(2\pi y_j\sqrt{(1+u_{j-1}^2)(1+u_j^{-2})}\right)^2
\prod_{j=1}^nu_j^{-k_j}\frac{du_j}{u_j}\\
    &\ll_{k_1,\dots,k_n}
    \exp\left(-\pi nr-\frac2{T_\mu}\sum_{j=1}^n y_j\right)r^{\sum_{j=1}^n jk_j-2n/3}\prod_{j=1}^ny_j^{-\sum_{m=j}^nk_m}.
\end{align*}
\end{lemma}
\begin{proof}
We prove by induction on $n$. If $n=1$, then   Lemma~\ref{lem bound integral of Bessel function} 
gives the result.
Assume now that $n>1$ and the statements hold for $n-1$. We consider the integral over $u_n$. Again by Lemma \ref{lem bound integral of Bessel function}, we have
\[
\int_0^\infty K_{ir}\left(2\pi y_n\sqrt{(1+u_{n-1}^2)(1+u_n^{-2})}\right)^2u_n^{-k_n}\frac{du_n}{u_n}
\ll_{k_n}e^{-\pi r-\frac{2y_n}{T_\mu}}r^{k_n-\frac23}y_n^{-k_n}(1+u_{n-1}^2)^{-\frac{k_n}{2}}
\]
since by assumption $k_n>0$.

Now
we use $(1+u_{n-1}^2)^{-k_n/2}\leq u_{n-1}^{-k_n}$, and it remains
to estimate the integral 
\begin{align*}
I_{n-1}=\int_{\R_{>0}^{n-1}}&K_{ir}\left(2\pi y_1\sqrt{1+u_i^{-2}}\right)^2\prod_{j=2}^{n-1}K_{ir}\left(2\pi y_i\sqrt{(1+u_{j-1}^2)(1+u_j^{-2})}\right)^2\prod_{j=1}^{n-1}u_j^{-k_j'}\frac{du_j}{u_j},
\end{align*}
where $k_j'=k_j$ for $j=1,\dots,n-2$, while $k'_{n-1}=k_{n-1}+k_n$.  By
our original assumption we have 
$\sum_{m=j}^{n-1}k_m'>0$
for every $1\leq j\leq n-1$, so 
 the induction hypothesis gives
\begin{align*}
    I_{n-1}&\ll_{k_1,\dots,k_{n-1}}
    e^{-\pi(n-1)r-\frac2{T_\mu}\sum_{j=1}^{n-1}y_j}r^{\sum_{j=1}^{n-1} jk_j+(n-1)k_n-2(n-1)/3}\prod_{j=1}^{n-1}y_j^{-\sum_{m=j}^{n-1}k_m-k_n}
\end{align*}
and the statement follows.
\end{proof}

Since our earlier 
assumptions imply $js+\gamma/2>0$ for $2\leq j\leq n-2$, we can
apply the
previous lemma
in \eqref{eq:I_L_bound2} to get
\begin{align}\label{eq:I_L_final_bound2}
I_L&\ll_{n,s,\gamma} B_{n-2}\cdot 
\exp\left(-\pi (n-3)r-\frac{2}{T_\mu}\sum_{j=1}^{n-3} y_j\right)
    r^{\frac{n(n-3)}{2}s+\frac {n-3}{2}\gamma-\frac{2(n-3)}{3}}
    \prod_{j=1}^{n-3}y_j^{-(n-1-j)s-\frac{\gamma}{2}}.
\end{align}

\subsection{Conclusion of the proof}\label{sec:conclusion_of_the_proof}
Now we combine the results of the previous sections to obtain the
bound on the Whittaker function. Recall that
$\abs{W_\mu^{\ast(n)}(y_1,\dots,y_{n-1})}^2\ll_n I_L\cdot I_R$,
hence Proposition~\ref{prop:right_factor} together with \eqref{eq:B_n-2}
and \eqref{eq:I_L_final_bound2} imply the following: if 
$s,\gamma>0$, then
\begin{align*}
    W_\mu^{\ast(n)}(y_1,\dots,y_{n-1})
    &\ll_{n,s,\gamma}\exp\left(-\frac{\pi (n-1)r}{2}-\frac{1}{T_\mu}\sum_{j=1}^{n-1} y_j\right)
\prod_{k,l=2}^{n-1}\Gamma_\R(s+\mu_k-\mu_l)^{1/2}\\
    &\qquad\qquad\qquad\qquad\qquad\quad\times
    r^{\frac{(n+1)(n-2)}{4}s+\frac {n-1}{4}\gamma-\frac{n-1}{3}}
    \prod_{j=1}^{n-1}y_j^{-\frac{(n-2)s}{2}-\frac{\gamma}{4}}.
\end{align*}
Finally, we write $\mathcal W_\mu^{(n)}(\diag(y_1\dots y_{n-1},\dots,y_{n-1},1))$ as
\[
W_\mu^{\ast(n)}(y_1,\dots,y_{n-1})\prod_{j=1}^{n-1}y_j^{\frac{j(n-j)}2}\prod_{k=1}^jy_j^{\frac{\mu_k+\mu_{n+1-k}}{2}}\prod_{1\leq k<l\leq n}\Gamma_\R(1+\mu_k-\mu_l)^{-1}.
\]
We consider first the Gamma factors. By Stirling's formula, we have
$$\abs{\Gamma_\R(\sigma+it)}\sim_\sigma\sqrt{2\pi}e^{-\pi\abs t/4}\abs t^{\frac{\sigma-1}2}$$
for $t\to\pm\infty$. Then
\begin{align*}
&e^{-\pi(n-1)r/2}\prod_{k,l=2}^{n-1}\Gamma_\R(s+\mu_k-\mu_l)^{1/2}\prod_{1\leq k<l\leq n}\Gamma_\R(1+\mu_k-\mu_l)^{-1}\\
    &\ll_{n,s}\left(\prod_{k=1}^{n-1}e^{\pi(\mu_k-\mu_n)/4}\right)\left(\prod_{l=2}^{n-1}e^{\pi(\mu_1-\mu_l)/4}\right)\left(\prod_{\substack{k,l=2\\j\neq k}}^{n-1}\abs{\mu_k-\mu_l}^{\frac{s-1}4}\right)e^{-\pi(n-1)(\mu_1-\mu_n)/4}\\
    &\ll_{n,s}
    \prod_{\substack{k,l=2\\j\neq k} }\abs{\mu_k-\mu_l}^{\frac{s-1}4}.
\end{align*}
Since we assume that the parameter $\mu$ is well-spaced, i.e., $\abs{\mu_k-\mu_l}\asymp T_\mu$
for any $k\neq l$,
this is bounded by $T_\mu^{(n-2)(n-3)(s-1)/4}$. 

We also have $r \asymp T_\mu$ and all this yields the bound
\begin{align*}
e^{-\frac{1}{T_\mu}\sum_{j=1}^{n-1} y_j}
\prod_{j=1}^{n-1}\left(\frac{T_\mu}{y_j}\right)^{\frac{n-2}{2}s+\frac{\gamma}{4}}T_\mu^{-\frac{3n^2-11n+14}{12}}
\prod_{j=1}^{n-1}y_j^{\frac{j(n-j)}2}.
\end{align*}
Now choosing small $s$ and $\gamma$ and 
taking into account the comment 
after Theorem~\ref{thm:main_thm} as well,
the proof of Theorem~\ref{thm:main_thm} is complete.

\bibliographystyle{alpha}
\bibliography{ref}

\end{document}